\documentclass[11pt,reqno]{amsart}
\usepackage{indentfirst, amssymb,amsmath,amsthm, mathrsfs,cite,graphicx,float}
\newtheorem*{theoA}{Theorem A}

\newtheorem{theo}{Theorem}[section]
\newtheorem{lem}{Lemma}[section]

\newtheorem{cor}{Corollary}[section]

\newtheorem{rem}{Remark}[section]

\newtheorem{open problem}{Open problem}[section]
\newcommand{\pa}{\partial}
\newcommand{\ol}{\overline}

\newcommand{\be}{\begin{equation}}
\newcommand{\ee}{\end{equation}}
\newcommand{\bs}{\begin{small}}
\newcommand{\es}{\end{small}}
\newcommand{\beas}{\begin{eqnarray*}}
\newcommand{\eeas}{\end{eqnarray*}}
\newcommand{\bea}{\begin{eqnarray}}
\newcommand{\eea}{\end{eqnarray}}
\renewcommand{\epsilon}{\varepsilon}
\numberwithin{equation}{section}
\begin{document}

\title[Pre-Schwarzian and Schwarzian norm estimates]{Pre-Schwarzian and Schwarzian norm estimates for certain classes of analytic and harmonic mappings}
\author[V. Allu, R. Biswas and R. Mandal]{Vasudevarao Allu, Raju Biswas and Rajib Mandal}
\date{}
\address{Vasudevarao Allu, Department of Mathematics, School of Basic Science, Indian Institute of Technology Bhubaneswar, Bhubaneswar-752050, Odisha, India.}
\email{avrao@iitbbs.ac.in}
\address{Raju Biswas, Department of Mathematics, Raiganj University, Raiganj, West Bengal-733134, India.}
\email{rajubiswasjanu02@gmail.com}
\address{Rajib Mandal, Department of Mathematics, Raiganj University, Raiganj, West Bengal-733134, India.}
\email{rajibmathresearch@gmail.com}
\let\thefootnote\relax
\begin{abstract} Let $\mathcal{A}$ denote the class of all analytic functions $f$ in the unit disk $\mathbb{D}:=\{z\in\mathbb{C}: |z|<1\}$ such that $f(0)=f'(0)-1=0$.
In this paper, we introduce a new subclass $\mathcal{C}_\theta(\gamma)$ of $\mathcal{A}$ consisting of functions $f$ that satisfy the relation 
\beas \textrm{Re}\left(e^{i\theta}\left(1+\frac{zf''(z)}{f'(z)}\right)\right)<\left(1+\frac{\gamma}{2}\right)\cos\theta,~ z\in\mathbb{D},~ \gamma>0, ~\text{and}~|\theta|<\frac{\pi}{2},\eeas
and investigate the Schwarzian derivative and Schwarzian norm for functions $f$ belonging to the class $\mathcal{C}_\theta(\gamma)$. 
We establish sharp estimates for the Schwarzian norm $\|S_f\|$ of functions $f$ in the class $\mathcal{C}_{\theta}(\gamma)$ 
and derive univalence criteria using both pre-Schwarzian and Schwarzian norm estimates. We also introduce a 
corresponding harmonic class $\mathcal{HC}_{\theta}(\gamma)$ consisting of mappings $f = h+\overline{g}$ with $h\in\mathcal{C}_{\theta}(\gamma)$ 
and dilatation $\omega=g'/h'\in\mathrm{Aut}(\mathbb{D})$. For this harmonic class, we derive bounds for both the pre-Schwarzian and Schwarzian norms, including sharp results in special cases.
\end{abstract}
\maketitle
\section{Introduction}
\noindent An analytic function $f(z)$ in a domain $\Omega\subseteq\mathbb{C}$ is said to be locally univalent if for each $z_0\in\Omega$, there exists a neighborhood $N(z_0)$ of 
$z_0$ such that $f(z)$ is univalent in $N(z_0)$. The Jacobian of a complex-valued function $f=u+i v$ is defined by $J_f(z)=|f_z|^2 - |f_{\overline{z}}|^2$. 
It is well-known that the non-vanishing of the Jacobian is necessary and sufficient conditions for local univalence of analytic mappings (see \cite[Chapter 1]{D1983}). 
Let $\mathcal{S}$ denote the class of all analytic and univalent function $f$ in $\mathbb{D}$ normalized by $f(0)=f'(0)-1=0$.  Thus, every function $f$ in $\mathcal{S}$ has the following Taylor series expansion:
\bea\label{e2} f(z)=z+\sum_{n=2}^\infty a_n z^n.\eea
\indent Let $\mathcal{B}$ be the class of all analytic functions $\omega:\mathbb{D}\rightarrow\mathbb{D}$ and $\mathcal{B}_0=\{\omega\in\mathcal{B} : \omega(0)=0\}$. 
Functions in $\mathcal{B}_0$ are called Schwarz function. In view of the Schwarz's lemma, if $\omega\in\mathcal{B}_0$, then $|\omega(z)|\leq |z|$ and $|\omega'(0)|\leq 1$. 
Strict inequality holds in both estimates unless $\omega(z)=e^{i\theta}z$, $\theta\in\mathbb{R}$. For $\omega\in\mathcal{B}$, the Schwarz-Pick lemma, 
gives the estimate  $|\omega'(z)|\leq (1-|\omega(z)|^2)/(1-|z|^2)$ for $z\in\mathbb{D}$.  In 1931, Dieudonn\'e \cite{D1931} first obtained the exact region of variability of $\omega'(z_0)$ for a fixed $z_0\in\mathbb{D}$ for functions in the class $\mathcal{B}_0$.
\begin{lem}\label{lem1} \cite{D1931} \cite[Dieudonn\'e's Lemma, P. 198]{D1983} Let $\omega\in \mathcal{B}_0$ and $z_0\not=0$ be a fixed point in $\mathbb{D}$. The region of variability of $\omega'(z_0)$ is given by
\beas \left|\omega'(z_0)-\frac{\omega(z_0)}{z_0}\right|\leq \frac{|z_0|^2-\left|\omega(z_0)\right|^2}{|z_0|(1-|z_0|^2)}.\eeas
The equality occurs if, and only if, $\omega$ is a Blaschke product of degree $2$ fixing $0$.
\end{lem}
\begin{lem}\cite{DP2008}\label{lem2} Suppose $f$ is analytic in $\mathbb{D}$ with $|f(z)|\leq1$, then we have
\beas \frac{\left|f^{(n)}(z)\right|}{n!}\leq \frac{1-|f(z)|^2}{(1-|z|)^{n-1}(1-|z|^2)}\quad\text{and}\quad |a_n|\leq 1-|a_0|^2\quad\text{for}\quad n\geq 1,~~|z|<1.\eeas\end{lem}
The Dieudonn\'e's lemma is an improvement of the Schwarz's lemma as well as the Schwarz-Pick lemma. A Blaschke product of degree $n\in\mathbb{N}$ is
 of the form 
 \beas f(z)=e^{i\theta}\prod_{j=1}^n \frac{z-z_j}{1-\ol{z_j}z},\quad z, z_j\in\mathbb{D}, ~ \theta\in\mathbb{R}. \eeas
\indent One of the most important and useful tools in geometric function theory is the differential subordination technique. With the help of differential subordination, many problems in geometric function theory can be solved in a simple and precise manner.
An analytic function $f$ in the unit disk $\mathbb{D}$ is said to be subordinate to an analytic function $g$ in $\mathbb{D}$, written as $f\prec g$, if there exists a function $\omega\in\mathcal{B}_0$ such that $f(z)=g(\omega(z))$ for $z\in\mathbb{D}$. Moreover, if $g$ is univalent in $\mathbb{D}$,
then $f \prec g$ if, and only if, $f(0)=g(0)$ and $f(\mathbb{D})\subseteq g(\mathbb{D})$. For basic details and results on subordination classes, we refer
to \cite[Chapter 6]{D1983}. 
Using the notion of subordination, Ma and Minda \cite{MM1992} have introduced more general subclasses of starlike and convex functions as follows:
\beas\mathcal{S}^*(\varphi)=\left\{f\in\mathcal{S}:\frac{zf'(z)}{f(z)}\prec\varphi(z)\right\}\quad\text{and}\quad\mathcal{C}(\varphi)=\left\{f\in\mathcal{S}:1+\frac{zf''(z)}{f'(z)}\prec\varphi(z)\right\},\eeas
where the function $\varphi :\mathbb{D}\to\mathbb{ C}$, called Ma-Minda function, is analytic and univalent in $\mathbb{D}$ such that $\varphi(\mathbb{D})$ has positive real 
part, symmetric with respect to the real axis, starlike with respect to $\varphi(0)=1$ and $\varphi'(0)>0$. A Ma-Minda function $\varphi(z)$ has the Taylor series expansion 
of the form $\varphi(z)=1+\sum_{n=1}^\infty a_nz^n $ $(a_1>0)$.
We call $\mathcal{S}^*(\varphi)$ and $\mathcal{C}(\varphi)$ the Ma-Minda type starlike and Ma-Minda type convex classes associated with $\varphi$, respectively. One 
can easily prove the inclusion 
relations $\mathcal{S}^*(\varphi)\subset\mathcal{S}^*$ and $\mathcal{C}(\varphi)\subset\mathcal{C}$. It is known that $f\in\mathcal{C}(\varphi)$ if, and only if, $zf'\in\mathcal{S}^*(\varphi)$.\\[2mm]
\indent For different choices of \(\phi\), the classes \(\mathcal{S}^{*}(\phi)\) and \(\mathcal{C}(\phi)\) generate several important subclasses, including the Janowski classes and strongly starlike and convex classes (see \cite{S1966, J1973}). In 2025, Kumar \cite{K2025} considered the following class
\beas S_\theta(\gamma)=\left\{f\in\mathcal{A}: \textrm{Re}\left(e^{i\theta}\frac{zf'(z)}{f(z)}\right)<\left(1+\frac{\gamma}{2}\right)\cos\theta,~ z\in\mathbb{D},~ \gamma>0~\text{and}~|\theta|<\frac{\pi}{2}\right\}.\eeas
It is evident that the function $f(z)=z$ belongs to the class $S_\theta(\gamma)$, and thus the class $S_\theta(\gamma)$ is non-empty. In 1973, Shah \cite{S1973} introduced the class $S_\theta(\gamma)$ for  $\theta=0$ and $\gamma=1$. Furthermore, Shah \cite{S1973} has proved that $S_0(1)$ is the class of starlike functions in one direction. In this paper, we consider a new class $\mathcal{C}_\theta(\gamma)$ defined as follows:
  \beas \mathcal{C}_\theta(\gamma)= \left\{f\in\mathcal{A}: \textrm{Re}\,\left(e^{i\theta}\left(1+\frac{zf''(z)}{f'(z)}\right)\right)<\left(1+\frac{\gamma}{2}\right)\cos\theta,~ z\in\mathbb{D},~ \gamma>0, ~|\theta|<\frac{\pi}{2}\right\}.\eeas 
The class $\mathcal{C}_{\theta}(\gamma)$ imposes that the function $p(z)=1+zf''(z)/f'(z)$ lies in a half-plane rotated by $\theta$, with $\gamma$ determining its distance from the origin. Two special cases are immediate:
\begin{itemize}
\item[(i)] For $\theta=0$, the class $\mathcal{C}_{0}(\gamma)$ coincides with the classical condition $\textrm{Re}\,(p(z)) < 1+\gamma/2$ studied by Ozaki 
\cite{O1941} for $\gamma=1$ and later investigated by Umezawa \cite{U1952}, who showed such functions are convex in one direction.
\item[(ii)] For $\theta \neq 0$, the half-plane $\textrm{Re}\,\left(e^{i\theta}p(z)\right) < (1+\gamma/2)\cos\theta$ is not symmetric about the real axis, placing the class outside the Ma--Minda framework.\end{itemize}
The parameters $\theta$ and $\gamma$ thus provide a controlled setting for studying how rotational asymmetry and domain translation affect Schwarzian‑derivative estimates.
In terms of subordination, the class $C_\theta(\gamma)$  can be defined as follows:
\bea\label{a1} f\in \mathcal{C}_\theta(\gamma)\Longleftrightarrow 1+\frac{zf''(z)}{f'(z)}\prec \frac{1-e^{-i\theta}\left(\gamma\cos \theta+e^{i\theta}\right)z}{1-z}:=g(z).\eea
\begin{figure}[H]
\begin{minipage}[c]{0.5\linewidth}
\centering
\includegraphics[scale=0.6]{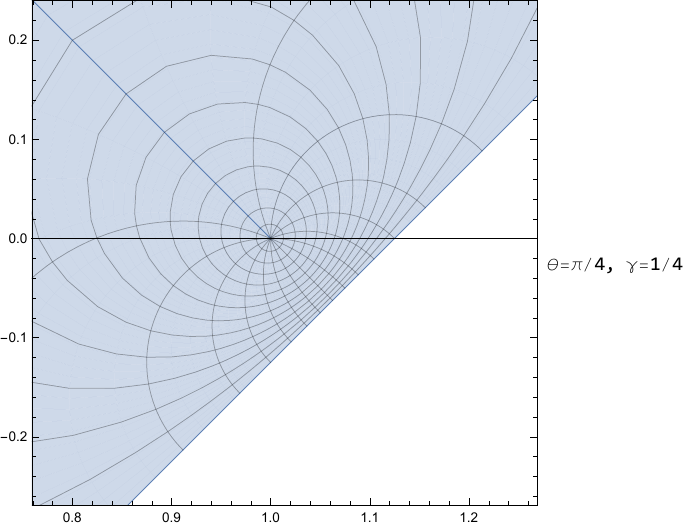}
\end{minipage}
\begin{minipage}[c]{0.49\linewidth}
\centering
\includegraphics[scale=0.6]{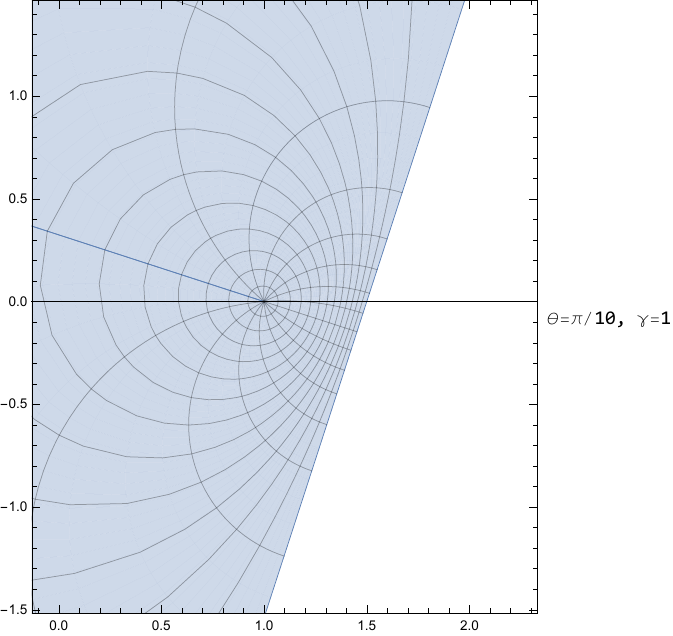}
\end{minipage}
\begin{minipage}[c]{0.49\linewidth}
\centering
\includegraphics[scale=0.6]{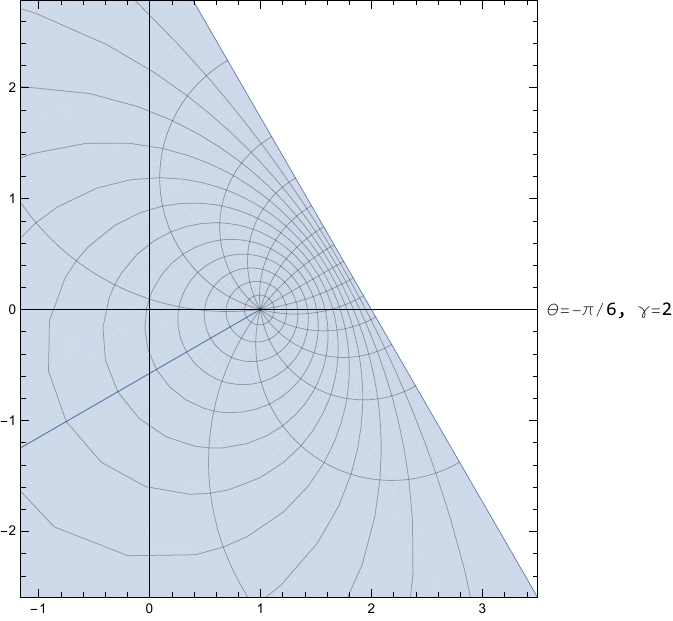}
\end{minipage}
\begin{minipage}[c]{0.49\linewidth}
\centering
\includegraphics[scale=0.6]{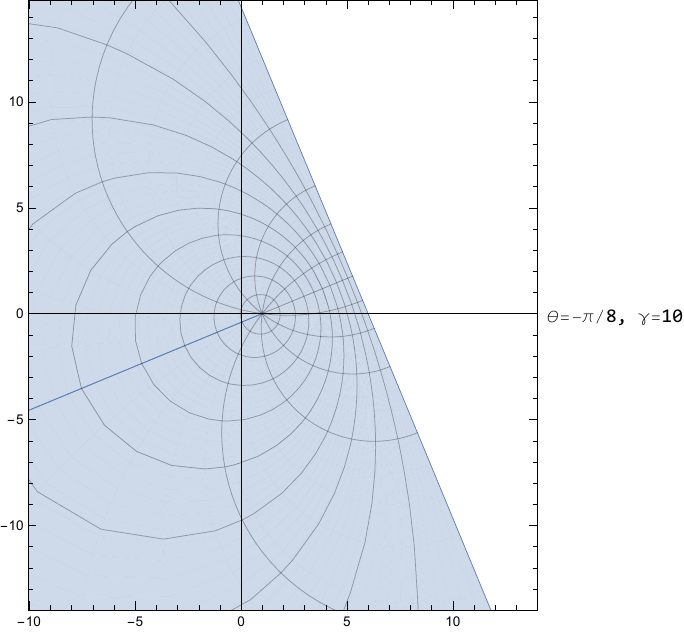}
\end{minipage}
\caption{Image of $\mathbb{D}$ under the mapping $g(z)$ for different values of $\theta$ and $\gamma$}
\label{Fig1}
\end{figure}
\noindent For $\theta=0$, the class $C_\theta(\gamma)$ reduces to the following class 
\beas \mathcal{G}(\gamma):=\left\{f\in\mathcal{A}: \textrm{Re}\left(1+zf''(z)/f'(z)\right)<1+\frac{\gamma}{2},~ z\in\mathbb{D},~ \gamma>0\right\}.\eeas
In 1941, Ozaki \cite{O1941} introduced the class $\mathcal{G} := \mathcal{G}(1)$ and showed that the functions in the class $\mathcal{G}$ are univalent in $\mathbb{D}$. 
Furthermore, functions in the class $\mathcal{G}$ are starlike in the unit disk $\mathbb{D}$ (see \cite{JO1995, PR1995}). Thus, the class $\mathcal{G}(\gamma)$ is 
included in $\mathcal{S}^*$ whenever $\gamma\in (0,1]$. One can easily show that functions in $\mathcal{G}(\gamma)$ are not univalent in $\mathbb{D}$ for $\gamma> 1$. For $0 <\gamma\leq 2/3$,
 the class $\mathcal{G}(\gamma)$ was studied by Uralegaddi {\it et al.} \cite{UGS1994}. 
In Figure \ref{Fig1}, we obtain the image of the unit disk $\mathbb{D}$ under the function $g(z)=\left(1-e^{-i\theta}\left(\gamma\cos \theta+e^{i\theta}\right)z\right)/(1-z)$ for different values of $\theta$ and $\gamma$ satisfying $|\theta|<\pi/2$ and $\gamma>0$.
\noindent Thus, it is easy to conclude that for $\theta\not=0$, the class $\mathcal{C}_\theta(\gamma)$ is a non Ma-Minda type of convex class associated with $g(z)=\left(1-e^{-i\theta}\left(\gamma\cos \theta+e^{i\theta}\right)z\right)/(1-z)$.\\[2mm]
\indent Let $f = u+i v$ be a complex-valued function of $z =x +i y$ in a simply connected domain $\Omega$. If $f\in C^2(\Omega)$ (continuous first and second order partial derivatives in 
 $\Omega$ exist) and satisfies the Laplace equation $\Delta f = 4 f_{z\ol{z}}=0$ in $\Omega$, then $f$ is said to be harmonic in $\Omega$, where 
$f_z=(1/2)(\pa f/\pa x -i \pa f/\pa y)$ and $f_{\ol{z}}=(1/2)(\pa f/\pa x +i \pa f/\pa y)$. Note that every harmonic
 mapping $f$ has the canonical representation $f =h +\ol{g}$, where $h$ and $g$ are analytic in $\Omega$, known respectively as the analytic and co-analytic parts of $f$. 
 This representation is unique up to an additive constant (see \cite{D2004}). The inverse function theorem and a result of Lewy \cite{L1936} shows that a harmonic
 function $f$ is locally univalent in $\Omega$ if, and only if, $J_f(z)\not= 0$ in $\Omega$. A harmonic mapping $f$ is locally univalent
 and sense-preserving in $\Omega$ if, and only if, $J_f(z)> 0$ in $\Omega$. Note that $|f_z|\not=0$ whenever $J_f>0$. Let $\mathcal{H}$ denote the class of all harmonic mappings $f=h+\ol{g}$ in $\mathbb{D}$ such that $h(0)=h'(0)-1=g(0)=0$. Thus, every $f\in\mathcal{H}$ has the following representation:
 \beas f(z)=z+\sum_{n=2}^\infty a_n z^n+\ol{\sum_{n=1}^\infty b_n z^n},\quad z\in\mathbb{D}.\eeas
 Let $\mathcal{S}_\mathcal{H}$ denote the subclass of $\mathcal{H}$ that are sense-preserving and univalent in the unit disk $\mathbb{D}$ and further
 let $S^0_{\mathcal{H}}=\{f =h+\ol{g}\in\mathcal{S}_{\mathcal{H}} : g'(0)=0\}$. We refer to \cite{D2004} for an in-depth study on planar harmonic univalent mappings.\\[2mm]
\indent In this paper, we now introduce a new class $\mathcal{HC}_\theta(\gamma)$ of harmonic mappings $f=h+\ol{g}$ in $\mathcal{H}$ such that
\beas \textrm{Re}\left(e^{i\theta}\left(1+\frac{z h''(z)}{h'(z)}\right)\right)<\left(1+\frac{\gamma}{2}\right)\cos\theta\quad \text{for}~ z\in\mathbb{D},~ \gamma>0, ~|\theta|<\frac{\pi}{2}\eeas
and the dilatation $\omega=g'/h'$ of $f$ belongs to $\text{Aut}(\mathbb{D})$, where $\text{Aut}(\mathbb{D})$ denotes the collection of all automorphisms of $\mathbb{D}$, {\it i.e.,} 
\beas \text{Aut}(\mathbb{D})=\left\{e^{i\theta}\frac{z-\alpha}{1-\ol{\alpha}z}: \alpha\in\mathbb{D},~\theta\in\mathbb{R}\right\}.\eeas
\subsection{Pre-Schwarzian and Schwarzian derivatives of analytic functions}
For a locally univalent analytic function $f$ defined in a simply connected domain $\Omega\subset \mathbb{C}$, the pre-Schwarzian derivative $P_f$ and the Schwarzian derivative $S_f$ of $f$ are, respectively, defined as follows:
\bea\label{c2} P_f(z)=\frac{f''(z)}{f'(z)}\quad\text{and}\quad S_f(z) = P_f'(z)-\frac{1}{2}P_f^2(z)=\frac{f'''(z)}{f''(z)}-\frac{3}{2}\left(\frac{f''(z)}{f'(z)}\right)^2.\eea
Moreover, the pre-Schwarzian and the Schwarzian norms of $f$ are, respectively, given by
\bea\label{r1} \Vert P_f\Vert= \sup_{z \in \mathbb{D}}(1-|z|^2)|P_f(z)|\quad\text{and}\quad \Vert S_f\Vert = \sup_{z \in \mathbb{D}}(1-|z|^2)^2|S_f(z)|.\eea
Some important global univalence criteria for a locally univalent analytic function have been obtained using the pre-Schwarzian and Schwarzian norms.
For a univalent analytic function $f$ in $\mathbb{D}$, it is well-known that $\Vert P_f\Vert\leq 6$ and the equality 
is attained for the Koebe function or its rotation. 
One of the most used univalence criterion for locally univalent analytic functions is the Becker's univalence criterion \cite{B1972}, 
which states that if $f$ is a locally univalent analytic function and $\sup_{z\in\mathbb{D}}\left(1-|z|^2\right) \left|zP_f(z)\right|\leq1$, then $f$ is univalent in $\mathbb{D}$. In a subsequent study, 
Becker and Pommerenke \cite{BP1984} have proved that the constant $1$ is sharp.
In 1976, Yamashita \cite{Y1976} 
proved that $\Vert P_f \Vert$ is finite if, and only if, $f$ is uniformly locally univalent in $\mathbb{D}$. Moreover, if $\Vert P_f\Vert<2$, then $f$ is bounded in $\mathbb{D}$ (see \cite{KS2002}).\\[2mm] 
\indent In terms of the Schwarzian derivative, it is well-known that for any univalent analytic function $f$ in $\mathbb{D}$, we have the sharp inequality $\Vert S_f\Vert \leq 6$ and
the equality is attained for the Koebe function or its rotation (see \cite{K1932}). 
In 1949, Nehari \cite{N1949} established important criteria for global univalence, expressed in terms of the Schwarzian derivative, by virtue of the connection with linear differential equations.
For instance, if $f$ is locally univalent and analytic in the unit disk $\mathbb{D}$ and satisfies $\Vert S_f\Vert \leq2$, then $f$ is univalent in $\mathbb{D}$. The bound $2$ is sharp \cite{H1949}. This result is known as Nehari's criterion for univalence.\\[2mm]
\indent In the field of univalent function theory, several researchers have studied the pre-Schwarzian norm for various subclasses of analytic and univalent functions.
In 1998, Sugawa \cite{S1998} established the sharp estimate of the pre-Schwarzian norm for functions in the class of strongly starlike functions of order $\alpha$ ($0<\alpha\leq 1$).
In $1999$, Yamashita \cite{Y1999} proved that $\Vert P_f\Vert \leq 6-4\alpha$ for $f\in\mathcal{S}^*(\alpha)$ and $\Vert P_f\Vert \leq 4(1-\alpha)$ for $f\in\mathcal{C}(\alpha)$, 
where $0\leq \alpha<1$ and both the estimates are sharp. In $2000$, Okuyama \cite{O2000} established the sharp estimate of the pre-Schwarzian norm for $\alpha$-spirallike functions. Kim and Sugawa \cite{KS2006} established the sharp 
estimate of the pre-Schwarzian norm $\Vert P_f\Vert \leq 2(A-B)/(1+\sqrt{1-B^2})$ for $f\in\mathcal{C}[A,B]$ (see also \cite{PS2008}). Ponnusamy and Sahoo \cite{PS2010} 
have obtained the sharp estimates of the pre-Schwarzian norm for functions in the class
$\mathcal{S}^*[\alpha,\beta]:=S^*\left(\left((1+(1-2\beta)z)/(1-z)\right)^\alpha\right)$, where $0<\alpha\leq 1$ and $0\leq \beta<1$. In $2014$, Aghalary and Orouji \cite{AO2014} obtained the sharp estimate of the pre-Schwarzian norm for $\alpha$-spirallike function of order $\rho$, where $\alpha\in(-\pi/2,\pi/2)$ and $\rho\in[0,1)$. The pre-Schwarzian norm of certain integral transform of $f$ for certain subclass of $f$ has been also studied in the literature. For a detailed study on pre-Schwarzian norm, we refer to \cite{KPS2004,PPS2008,PS2008, AP2023,1AP2023,2AP2023,AP2024,CKPS2005, ABM2025, B2025} and the references therein.
\subsection{Pre-Schwarzian and Schwarzian norms of harmonic mappings}
\indent For a locally univalent harmonic mapping $f=h+\overline{g}$ in the unit disk $\mathbb{D}$, Hern\'andez and Mart\'in \cite{HM2015} have defined the pre-Schwarzian and Schwarzian derivatives, respectively, as follows:
\beas
P_f &=& \left(\log(J_f)\right)_z = \frac{h''}{h'}-\frac{\overline{\omega}\omega'}{1-|\omega|^2}\quad \text{and}\hfill\\[2mm]
S_f &=& \left(\log J_f\right)_{zz}-\frac{1}{2}\left(\log J_f\right)_z^2= S_h+\frac{\overline{\omega}}{1-|\omega|^2}\left(\frac{h''}{h'}\omega'-\omega''\right)-\frac{3}{2}\left(\frac{\omega'\overline{\omega}}{1-|\omega|^2}\right)^2,\eeas
where $S_h$ is the classical Schwarzian derivative of the analytic function $h$, $J_f$ is the Jacobian and $\omega=g'/h'$ is the second complex dilatation of $f$. These derivatives play a crucial role in understanding the behavior of harmonic functions, particularly in the context of conformal mappings and their geometric properties. 
This notion of pre-Schwarzian and Schwarzian derivatives of harmonic functions is a generalization of the classical pre-Schwarzian and Schwarzian derivatives of analytic functions. 
Note that when $f$ is analytic, we have $\omega=0.$ It is also easy to see that $S_f = (P_f)_z-(1/2)(P_f)^2.$ As in the case of analytic functions, for a sense-preserving 
locally univalent harmonic mapping $f = h+\bar{g}$ in the unit disk $\mathbb{D}$, the pre-Schwarzian norm $\Vert P_f\Vert$ and the Schwarzian norm $\Vert S_f\Vert$ are defined by (\ref{r1}). For a comprehensive study on the pre-Schwarzian and Schwarzian derivatives for harmonic mappings, we refer to \cite{HM2015, LP2018}.\\[2mm]
\indent Hern\'andez and Mart\'in \cite{1HM2015} have established a sufficient condition for a locally univalent harmonic mapping $f$ in the unit disk $\mathbb{D}$ to be univalent 
in $\mathbb{D}$. Moreover, they proved that if $f$ is a sense‑preserving harmonic mapping in $\mathbb{D}$ with $\Vert S_f\Vert \leq \delta t$ for some $t<1$ and its dilatation 
$\omega_f$ satisfies $\sup_{z\in\mathbb{D}} |\omega_f(z)| < 1$,  then $f$ can be extended to a quasiconformal mapping in the Riemann sphere $\mathbb{C}\cup\{\infty\}$. The sharp value of $\delta$ remains an unresolved problem.
The key results connecting the Schwarzian derivative and quasiconformal mappings are given in the following theorem.
\begin{theoA} \cite[Chapter II]{L1987}\cite[Theorem 3.2]{O1998} If $f$ extends to a $k$-quasiconformal $(0\leq k <1)$ mapping of the Riemann sphere $\mathbb{C}\cup\{\infty\}$, then $\Vert S_f\Vert\leq 6k$. Conversely, if $\Vert S_f\Vert\leq 2k$, then $f$ extends to a $K$-quasiconformal mapping of the Riemann sphere $\mathbb{C}\cup\{\infty\}$, where $K=(1+k)/(1-k)$.
\end{theoA}
In $2015$, Hern\'andez and Mart\'in \cite[Theorem 5]{HM2015} demonstrated the existence of a positive constant $C$ such that $\Vert S_f\Vert \leq C$ for all univalent harmonic 
mappings $f$ in the unit disk $\mathbb{D}$.  
In $2016$, Graf \cite[Theorem 3]{G2016} proved that $\| S_f\|$ is bounded for locally univalent harmonic mappings in the affine and linear invariant family.  
However, the sharp constants in these bounds are still unknown.
In 2016, Graf \cite[Theorem 3]{G2016} proved that $\Vert S_f\Vert$ is bounded for locally univalent harmonic mappings in the affine and linear invariant family.
In 2015, Hern\'andez and Mart\'in \cite[Theorem 5]{HM2015} have demonstrated that there exists a positive constant $C$ such that $\Vert S_f\Vert\leq C$ for all univalent
harmonic mappings $f$ in the unit disk $\Bbb{D}$. But inquiries regarding the sharp bound remain unresolved. Hern\'andez and Mart\'in \cite{HM2015} have proved the pre-Schwarzian norm $\Vert P_f\Vert \leq 5$ and Schwarzian norm $\Vert S_f\Vert\leq 6$ of convex harmonic functions. Here, the bound $6$ is non-sharp, while the bound $5$ is sharp.\\[2mm]
\noindent In this paper, we establish sharp estimates of the Schwarzian derivative and Schwarzian norm for functions in the class $\mathcal{C}_\theta(\gamma)$. 
Using pre-Schwarzian and Schwarzian norms, we determine under which conditions the functions in the class $\mathcal{C}_\theta(\gamma)$ are univalent in the unit disk $\mathbb{D}$. Moreover, we investigate the pre-Schwarzian and Schwarzian norm for functions in the class $\mathcal{HC}_\theta(\gamma)$.
\section{Main results}
In the following result, we establish the sharp estimate of the Schwarzian derivative for functions $f$ in the class $\mathcal{C}_\theta(\gamma)$.
\begin{theo}\label{Th1} For $|\theta|<\pi/2$ and $\gamma>0$, let $f$ be in the class $\mathcal{C}_\theta(\gamma)$. Then the Schwarzian derivative satisfies the following inequality
\beas |S_f| \leq \frac{\gamma\cos\theta\sqrt{4+(\gamma^2+4\gamma)\cos^2\theta}}{2(1-|z|)^2}.\eeas
Moreover, the equality occur for the function $f_1(z)$ defined by 
\beas f_1(z)=\dfrac{1-(1-z)^{\gamma e^{-i\theta}\cos\theta+1}}{\gamma e^{-i\theta}\cos\theta+1}.\eeas
\end{theo}
\begin{proof}
Let $f\in \mathcal{C}_\theta(\gamma)$, where $|\theta|<\pi/2$ and $\gamma>0$. By the definition of the class $\mathcal{C}_\theta(\gamma)$ and from (\ref{a1}), we have
\beas 1+\frac{zf''(z)}{f'(z)}\prec \frac{1-e^{-i\theta}\left(\gamma\cos \theta+e^{i\theta}\right)z}{1-z}.\eeas
Thus, there exist a Schwarz function $\omega\in\mathcal{B}_0$ such that
\beas &&1+\frac{zf''(z)}{f'(z)}=\frac{1-e^{-i\theta}\left(\gamma\cos \theta+e^{i\theta}\right)\omega(z)}{1-\omega(z)},\\[2mm]\text{\it i.e.,}
&&\frac{f''(z)}{f'(z)}=\frac{\left(1-e^{-i\theta}\left(\gamma\cos \theta+e^{i\theta}\right)\right)\omega(z)}{z(1-\omega(z))}.\eeas
For simplicity, let $\Gamma(\gamma, \theta)=-e^{-i\theta}\gamma\cos \theta$.
Therefore, the Schwarzian derivative for the function $f\in S_\theta(\gamma)$ is 
\beas S_f(z)&=&\left(\frac{f''(z)}{f'(z)}\right)'-\frac{1}{2}\left(\frac{f''(z)}{f'(z)}\right)^2\\[2mm]
&=&\Gamma(\gamma, \theta)\left(\frac{z\omega'(z)-\omega(z)}{z^2(1-\omega(z))^2}+\frac{(2-\Gamma(\gamma, \theta))\omega^2(z)}{2z^2(1-\omega(z))^2}\right)\\
&=&\Gamma(\gamma, \theta)\left(\frac{\omega'(z)-\omega(z)/z}{z(1-\omega(z))^2}+\frac{(2-\Gamma(\gamma, \theta))\omega^2(z)}{2z^2(1-\omega(z))^2}\right).\eeas
Let us consider the transformation $h(z)=\omega'(z)-\omega(z)/z$. In view of \textrm{Lemma \ref{lem1}}, the function $h(z)$ varies over the closed disk
\bea\label{a2} |h(z)|\leq \frac{|z|^2-|\omega(z)|^2}{|z|\left(1-|z|^2\right)}\quad\text{ for fixed}\quad |z|<1.\eea
Using (\ref{a2}), we have
\beas |S_f(z)|&=&\left|\Gamma(\gamma, \theta)\right|\left|\left(\frac{h(z)}{z(1-\omega(z))^2}+\frac{(2-\Gamma(\gamma, \theta))\omega^2(z)}{2z^2(1-\omega(z))^2}\right)\right|\\[2mm]
&\leq&\left|\Gamma(\gamma, \theta)\right| \left(\frac{|h(z)|}{|z|(1-|\omega(z)|)^2}+\frac{|2-\Gamma(\gamma, \theta)| |\omega(z)|^2}{2|z|^2(1-|\omega(z)|)^2}\right)\\[2mm]
&\leq&\left|\Gamma(\gamma, \theta)\right| \left(\frac{|z|^2-|\omega(z)|^2}{|z|^2(1-|z|^2)(1-|\omega(z)|)^2}+\frac{|2-\Gamma(\gamma, \theta)| |\omega(z)|^2}{2|z|^2(1-|\omega(z)|)^2}\right).\eeas
For $0<t:=|\omega(z)|\leq |z|<1$, we have
\beas |S_f(z)|&\leq& \left|\Gamma(\gamma, \theta)\right| \left(\frac{|z|^2-t^2}{|z|^2(1-|z|^2)(1-t)^2}+\frac{|2-\Gamma(\gamma, \theta)| t^2}{2|z|^2(1-t)^2}\right)\\[2mm]
&=& \left|\Gamma(\gamma, \theta)\right| \left(\frac{2|z|^2-2t^2+|2-\Gamma(\gamma, \theta)| (1-|z|^2)t^2}{2|z|^2(1-|z|^2)(1-t)^2}\right).\eeas
Therefore, we have
\bea \label{e7}|S_f(z)|\leq  \left|\Gamma(\gamma, \theta)\right| F_1(|z|,t),\eea
where 
\beas F_1(r, t)=\frac{2r^2+\left(|2-\Gamma(\gamma, \theta)| (1-r^2)-2\right)t^2}{2r^2(1-r^2)(1-t)^2}\quad\text{for}\quad r=|z|.\eeas
The objective is to determine the maximum of $F_1(r, t)$ on $\Omega=\{(r, t): 0< t\leq r<1\}$.
Differentiating partially $F_1(r, t)$ with respect to $t$, we obtain
\beas\frac{\pa}{\pa t}F_1(r, t)&=&\frac{4r^2+2\left(|2-\Gamma(\gamma, \theta)| (1-r^2)-2\right)t^2}{2r^2(1-r^2)(1-t)^3}+\frac{2\left(|2-\Gamma(\gamma, \theta)| (1-r^2)-2\right)t}{2r^2(1-r^2)(1-t)^2}\\[2mm]
&=&\frac{4r^2+2\left(|2-\Gamma(\gamma, \theta)| (1-r^2)-2\right)t^2+2\left(|2-\Gamma(\gamma, \theta)| (1-r^2)-2\right)t(1-t)}{2r^2(1-r^2)(1-t)^3}\\
&=&\frac{2r^2+\left(|2-\Gamma(\gamma, \theta)| (1-r^2)-2\right)t}{r^2(1-r^2)(1-t)^3}.\eeas
Note that $|2-\Gamma(\gamma, \theta)|=\sqrt{4+(\gamma^2+4\gamma)\cos^2\theta}>2$ for all $\gamma>0$ and $|\theta|<\pi/2$. For $\gamma>0$ and $|\theta|<\pi/2$, we have 
\beas |2-\Gamma(\gamma, \theta)| (1-r^2)-2>-2r^2~~\text{and}~~2r^2+\left(|2-\Gamma(\gamma, \theta)| (1-r^2)-2\right)t>2r^2(1-t)>0\eeas
for all $(r, t)\in\Omega$.
Thus, we have $\frac{\pa}{\pa t}F_1(r, t)\geq 0$ for all $(r, t)\in\Omega$, which shows that $F_1(r, t)$ is a monotonically increasing function of $t\in(0,r]$. Consequently, $F_1(r, t)\leq F_1(r, r)$ and hence, we have 
\beas |S_f|\leq\frac{\left|\Gamma(\gamma, \theta)\right| \left(|2-\Gamma(\gamma, \theta)|\right)}{2(1-r)^2}=\frac{\gamma\cos\theta\sqrt{4+(\gamma^2+4\gamma)\cos^2\theta}}{2(1-r)^2}.\eeas
\indent To show that the estimate is sharp, we consider the function $f_1$ given by
\bea\label{e1} f_1(z)=\dfrac{1-(1-z)^{\gamma e^{-i\theta}\cos\theta+1}}{\gamma e^{-i\theta}\cos\theta+1}\in \mathcal{C}_\theta(\gamma).\eea
The Schwarzian derivative of $f_1$ is given by
\beas &&S_{f_1}(z)=\left(\frac{f_1''(z)}{f_1'(z)}\right)'-\frac{1}{2}\left(\frac{f_1''(z)}{f_1'(z)}\right)^2=-\frac{\left(\gamma e^{-i\theta}\cos\theta+2\right)\gamma e^{-i\theta}\cos\theta}{2(1-z)^2},\\[2mm]\text{\it i.e.,}
&&\left|S_{f_1}(z)\right|=\frac{\left|\gamma e^{-i\theta}\cos\theta+2\right| \left|\gamma e^{-i\theta}\cos\theta\right|}{2|1-z|^2}.\eeas
On the positive real axis, we note that
\beas \left|S_{f_1}(z)\right|=\frac{\gamma \cos\theta \sqrt{4+(\gamma^2+4\gamma)\cos^2\theta}}{2(1-|z|)^2}.\eeas
This completes the proof.
\end{proof}
\begin{rem} Let $f\in\mathcal{C}_\theta(\gamma)$, then in view of \textrm{Theorem \ref{Th1}}, we have 
\beas \left|zP_f(z)\right|=\left|\frac{zf''(z)}{f'(z)}\right|=\left|\frac{-e^{-i\theta}\gamma\cos \theta~\omega(z)}{1-\omega(z)}\right|\leq \frac{\gamma\cos \theta ~|\omega(z)|}{1-|\omega(z)|}\leq \frac{\gamma\cos \theta ~|z|}{1-|z|}.\eeas
Thus, we have $\sup_{z\in\mathbb{D}}\left(1-|z|^2\right) \left|zP_f(z)\right|\leq 2\gamma\cos \theta$. 
According to Becker's univalence criterion, the function $f\in \mathcal{C}_\theta(\gamma)$ is univalent in the unit disk $\mathbb{D}$ for those values of $\gamma>0$ and $|\theta|<\pi/2$ such that the condition $2\gamma\cos \theta\leq 1$ holds. In the case of $\theta=0$, the functions belonging to the class $\mathcal{C}_\theta(\gamma)$ are univalent in 
$\mathbb{D}$, provided that $\gamma\leq 1/2$. It is evident that as $|\theta|\to \pi/2$, the range of values for $\gamma$ increases.  
\end{rem}
In the following result, we establish the sharp estimate of the Schwarzian norm for functions $f$ in the class $\mathcal{C}_\theta(\gamma)$.
\begin{theo}\label{Th2} For $|\theta|<\pi/2$ and $\gamma>0$, let $f$ be in the class $\mathcal{C}_\theta(\gamma)$. Then the Schwarzian norm satisfies the following inequality
\beas \Vert S_f\Vert \leq2\gamma\cos\theta\sqrt{4+(\gamma^2+4\gamma)\cos^2\theta}\eeas
and the estimate is best possible.
\end{theo}
\begin{proof} As $f$ is in the class $\mathcal{C}_\theta(\gamma)$ for $|\theta|<\pi/2$ and $\gamma>0$, in view of \textrm{Theorem \ref{Th1}}, the Schwarzian derivative of $f$ satisfies the following inequality
\beas |S_f| \leq \frac{\gamma\cos\theta\sqrt{4+(\gamma^2+4\gamma)\cos^2\theta}}{2(1-|z|)^2}.\eeas  
Therefore, the Schwarzian norm of $f$ is 
\beas \Vert S_f\Vert =\sup_{z\in\mathbb{D}}\left(1-|z|^2\right)^2\left|S_f(z)\right|\leq 2\gamma\cos\theta\sqrt{4+(\gamma^2+4\gamma)\cos^2\theta}.\eeas
\indent To show that the estimate is sharp, we consider the function $f_1$ given in (\ref{e1}) and thus, we have 
\beas \sup_{z\in\mathbb{D}}\left(1-|z|^2\right)^2\left|S_{f_1}(z)\right|=\sup_{z\in\mathbb{D}}\left(1-|z|^2\right)^2\frac{\gamma \cos\theta \sqrt{4+(\gamma^2+4\gamma)\cos^2\theta}}{2|1-z|^2}.\eeas
On the positive real axis, we have 
\beas \sup_{0\leq z<1}\left(1-z^2\right)^2\frac{\gamma \cos\theta \sqrt{4+(\gamma^2+4\gamma)\cos^2\theta}}{2(1-z)^2}=2\gamma \cos\theta \sqrt{4+(\gamma^2+4\gamma)\cos^2\theta}.\eeas
This completes the proof.
\end{proof}
If we put $\theta= 0$ in \textrm{Theorem \ref{Th2}}, we get the Schwarzian norm estimate for the class $G(\gamma)$, which was proved by Ali and Pal \cite{1AP2023}.
\begin{cor}
For $\gamma > 0$, let $f\in G(\gamma)$ be of the form (\ref{e2}). Then the Schwarzian norm satisfy the following sharp inequality
\beas \Vert S_f\Vert \leq2\gamma(\gamma+2).\eeas
\end{cor}
 Let $\delta(\gamma,\theta):=\gamma\cos\theta\sqrt{4+(\gamma^2+4\gamma)\cos^2\theta}$. In view of \textrm{Theorem \ref{Th2}} and \textrm{Theorem A}, we  immediately obtain the following result for functions in $C_\theta(\gamma)$. 
 \begin{cor}\label{cor1}
 Let $\gamma>0$ and $|\theta|<\pi/2$ be such that $\delta(\gamma,\theta)\in[0,1)$, and $f\in \mathcal{C}_\theta(\gamma)$. Then $f$ extends to an $(1+\delta(\gamma,\theta))/(1-\delta(\gamma,\theta))$-quasiconformal mapping.
 \end{cor}
\begin{rem}
In view of Nehari's criterion for univalence, the functions $f$ in the class $\mathcal{C}_\theta(\gamma)$ are univalent in $\mathbb{D}$ whenever $\delta(\gamma,\theta)\in[0,1)$, where $\gamma>0$ and $|\theta|<\pi/2$. 
\end{rem}
In Table \ref{tab0}, we obtain the values of $\gamma$ for certain values of $\theta$ in $(-\pi/2, \pi/2)$ for which $\delta(\gamma,\theta)\in[0,1)$. The values are obtained by solving $\delta(\gamma,\theta) < 1$.
\begin{table}[H]
\centering
\caption{Range of values for \(\gamma\) corresponding to certain values of \(\theta\) for which \(\delta(\gamma,\theta) \in [0,1)\). The values are obtained by solving \(\delta(\gamma,\theta) < 1\).}
\begin{tabular}{|c|c|}
\hline
\(\theta\) & Range of \(\gamma\) \\ \hline
\(\pm\pi/4\) & \(0 < \gamma < 0.608465\) \\ \hline
\(\pm\pi/5\) & \(0 < \gamma < 0.524526\) \\ \hline
\(\pm\pi/6\) & \(0 < \gamma < 0.486367\) \\ \hline
\(\pm\pi/10\) & \(0 < \gamma < 0.438151\) \\ \hline
\end{tabular}
\caption{The range of values for $\gamma$ corresponding to certain values of $\theta$ for which $\delta(\gamma,\theta)$ is in $[0, 1)$}
\label{tab0}\end{table}
In the following result, we obtain the estimate of the pre-Schwarzian norm for functions in the class $\mathcal{HC}_{\theta}(\gamma)$.
\begin{theo}\label{Th3} For $|\theta|<\pi/2$ and $\gamma>0$, let $f=h+\ol{g}$ be in the class $\mathcal{HC}_\theta(\gamma)$. Then the pre-Schwarzian norm satisfies the following inequality
\beas \Vert P_f\Vert \leq 1+2\gamma \cos\theta.\eeas
and the estimate is best possible whenever $\theta=0$.
\end{theo}
\begin{proof}
As $f=h+\ol{g}\in\mathcal{HC}_\theta(\gamma)$, thus, $h$ satisfies the following condition 
\beas \textrm{Re}\left(e^{i\theta}\left(1+\frac{z h''(z)}{h'(z)}\right)\right)<\left(1+\frac{\gamma}{2}\right)\cos\theta\quad \text{for}~ z\in\mathbb{D},~ \gamma>0~\text{and} ~|\theta|<\frac{\pi}{2}.\eeas
In view of (\ref{a1}), we have 
\beas 1+\frac{zh''(z)}{h'(z)}\prec \frac{1-e^{-i\theta}\left(\gamma\cos \theta+e^{i\theta}\right)z}{1-z}.\eeas
Thus, there exist a Schwarz function $\omega\in\mathcal{B}_0$ such that
\beas \frac{h''(z)}{h'(z)}=\frac{\Gamma(\gamma, \theta)\omega(z)}{z(1-\omega(z))}.\eeas
where $\Gamma(\gamma, \theta)=-e^{-i\theta}\gamma\cos \theta$. For $0<t:=|\omega(z)|\leq |z|<1$, we have 
\beas \left|\frac{h''(z)}{h'(z)}\right|\leq \frac{\gamma t \cos\theta }{|z|(1-t)}:=G_1(|z|, t).\eeas
Differentiating partially $G_1(|z|, t)$ with respect to $t$, we obtain
\beas \frac{\pa}{\pa t} G_1(|z|, t)=\frac{\gamma \cos\theta}{|z|(1-t)^2}>0,\eeas
which shows that $G_1(|z|, t)$ is a monotonically increasing function $t\in(0, |z|]$ and hence, we have $G_1(|z|, t)\leq G_1(|z|, |z|)$, {\it i.e.,}
\beas \left|\frac{h''(z)}{h'(z)}\right|\leq \frac{\gamma \cos\theta}{1-|z|}.\eeas
Therefore, the pre-Schwarzian norm for the function $f\in\mathcal{HC}_\theta(\gamma)$ is
\beas \Vert P_f\Vert=\sup_{z\in\mathbb{D}}(1-|z|^2)\left|\frac{h''}{h'}-\frac{\overline{\omega}\omega'}{1-|\omega|^2}\right|,\eeas
where $\omega=g'/h'$ is the dilatation of $f$. Since $\omega\in\text{Aut}(\mathbb{D})$, in view of the Schwarz-Pick lemma and triangle inequality, we have 
\beas \Vert P_f\Vert&\leq& \sup_{z\in\mathbb{D}}(1-|z|^2)\left(\left|\frac{h''}{h'}\right|+\left|\frac{\overline{\omega}\omega'}{1-|\omega|^2}\right|\right)\\
&\leq& \sup_{z\in\mathbb{D}}(1-|z|^2)\left(\frac{\gamma \cos\theta}{1-|z|}+\frac{|\overline{\omega}|}{1-|z|^2}\right)\\
&\leq& \sup_{z\in\mathbb{D}}\left(\gamma \cos\theta(1+|z|)+|\omega|\right)=1+2\gamma \cos\theta.\eeas
To prove the sharpness of the result, we consider the function $f_2=h_2+\ol{g_2}\in\mathcal{H}$ with dilatation $\omega(z)=(a-z)/(1-az)$, $a\in(0, 1)$ and $h_2$ be such that 
\beas 1+\frac{zh_2''(z)}{h_2'(z)}=\frac{1-e^{-i\theta}\left(\gamma\cos \theta+e^{i\theta}\right)z}{1-z}.\eeas 
 It is evident that $f_2\in\mathcal{HC}_\theta(\gamma)$ and $\omega'(z)=(a^2-1)/(1-az)^2$. Therefore, 
 \beas&&1-|\omega(z)|^2=1-\frac{(a-z)}{(1-az)}\frac{(a-\ol{z})}{(1-a\ol{z})}=\frac{(1-a^2)(1-|z|^2)}{(1-az)(1-a\ol{z})}\\\text{and}
 &&\frac{\overline{\omega}\omega'}{1-|\omega|^2}=\frac{(a^2-1)}{(1-az)^2}\cdot\frac{(a-\ol{z})}{(1-a\ol{z})}\cdot \frac{(1-az)(1-a\ol{z})}{(1-a^2)(1-|z|^2)}=\frac{\ol{z}-a}{(1-az)(1-|z|^2)}.\eeas
 Therefore, the pre-Schwarzian norm of $f_1$ is 
 \beas \Vert P_{f_2}\Vert &=&\sup_{z\in\mathbb{D}} (1-|z|^2)\left|\frac{h''}{h'}-\frac{\overline{\omega}\omega'}{1-|\omega|^2}\right|\\[2mm]
 &=&\sup_{z\in\mathbb{D}} (1-|z|^2)\left|\frac{\Gamma(\gamma, \theta)}{1-z}-\frac{\ol{z}-a}{(1-az)(1-|z|^2)}\right|,\eeas
 where $\Gamma(\gamma, \theta)=-e^{-i\theta}\gamma\cos \theta$.  Let $\theta=0$, then $\Gamma(\gamma, \theta)=-\gamma$.
 On the positive real axis, we have
 \beas \sup_{0\leq r<1} (1-r^2)\left|\frac{\Gamma(\gamma, \theta)}{1-r}-\frac{r-a}{(1-ar)(1-r^2)}\right|= \sup_{0\leq r<1} \left|\gamma(1+r)+\frac{r-a}{(1-ar)}\right|.\eeas 
 Let $F_2(r):=\gamma(1+r)+(r-a)/(1-ar)$. Therefore, 
 \beas F_2'(r)=\gamma+\frac{1-a^2}{(1-ar)^2}\geq 0,\eeas
 which shows that $F_2(r)$ is a monotonically increasing function of $r\in [0, 1)$ and hence, we have 
 $\Vert P_{f_2}\Vert =\sup_{0\leq r<1} F_2(r)=2\gamma+1$. This completes the proof. 
 \end{proof}  
 \begin{rem} Let $f=h+\ol{g}\in\mathcal{HC}_\theta(\gamma)$. Then, in view of \textrm{Theorem \ref{Th3}}, we have 
 \beas \left|\frac{h''(z)}{h'(z)}\right|\leq \frac{\gamma \cos\theta}{1-|z|},\quad \text{which shows that}\quad \Vert P_h\Vert \leq 2\gamma\cos\theta.\eeas
 Using the function given in (\ref{e1}), we easily show that the estimation is sharp. 
 Hence, we have $\left|\Vert P_f\Vert -\Vert P_h\Vert \right|\leq 1$ for functions $f=h+\ol{g}\in \mathcal{HC}_\theta(\gamma)$ and equality occurs when $\theta=0$. 
 Therefore, the class of functions $\mathcal{HC}_\theta(\gamma)$ supports the result of Liu and Ponnusamy \cite[Theorem 2.1]{LP2018}.
 \end{rem}
\begin{theo}\label{Th4}
For $|\theta|<\pi/2$ and $\gamma>0$, let $f=h+\ol{g}$ be in the class $\mathcal{HC}_\theta(\gamma)$. Then the Schwarzian norm satisfies the following inequality
\beas \Vert S_f\Vert \leq 2\gamma\cos\theta\sqrt{4+(\gamma^2+4\gamma)\cos^2\theta}+2\gamma\cos\theta +11/2.\eeas
\end{theo}
\begin{proof}
Let $f=h+\ol{g}\in\mathcal{HC}_\theta(\gamma)$, then in view of \textrm{Theorem \ref{Th3}}, we have 
\beas 1+\frac{zh''(z)}{h'(z)}\prec \frac{1-e^{-i\theta}\left(\gamma\cos \theta+e^{i\theta}\right)z}{1-z}.\eeas
Using the argument as in the proof of \textrm{Theorem \ref{Th1}}, we obtain the Schwarzian derivative of $h$ satisfies the following inequality
\beas |S_h| \leq \frac{\gamma\cos\theta\sqrt{4+(\gamma^2+4\gamma)\cos^2\theta}}{2(1-|z|)^2}.\eeas  
Again, in view of \textrm{Theorem \ref{Th3}}, we have 
\beas \left|\frac{h''(z)}{h'(z)}\right|\leq \frac{\gamma \cos\theta}{1-|z|}.\eeas 
Therefore, the Schwarzian derivative of $f=h+\ol{g}\in\mathcal{HC}_\theta(\gamma)$ is  
\beas&& S_f= S_h+\frac{\overline{\omega}}{1-|\omega|^2}\left(\frac{h''}{h'}\omega'-\omega''\right)-\frac{3}{2}\left(\frac{\omega'\overline{\omega}}{1-|\omega|^2}\right)^2\\\text{\it i.e.,}
&& |S_f(z)|\leq  |S_h|+\left|\frac{\overline{\omega}}{1-|\omega|^2}\left(\frac{h''}{h'}\omega'-\omega''\right)\right|+\left|\frac{3}{2}\left(\frac{\omega'\overline{\omega}}{1-|\omega|^2}\right)^2\right|.\eeas
Since $\omega\in\text{Aut}(\mathbb{D})$, in view of \textrm{Lemma \ref{lem2}}, we have 
\beas (1-|z|^2)^2\frac{\left|\omega(z)\omega''(z)\right|}{1-|\omega(z)|^2}\leq 2(1+|z|)|\omega(z)|.\eeas
In view of the Schwarz-Pick lemma, we have
\beas
(1-|z|^2)^2|S_f(z)|&\leq& (1-|z|^2)^2|S_h(z)|+(1-|z|^2)\frac{\left|\overline{\omega(z)}\omega'(z)\right|}{1-|\omega(z)|^2}\cdot (1-|z|^2)\left|\frac{h''(z)}{h'(z)}\right|\\[2mm]
&&+(1-|z|^2)^2\frac{\left|\ol{\omega(z)}\omega''(z)\right|}{1-|\omega(z)|^2}+\frac{3}{2} \left((1-|z|^2)\frac{|\omega'(z)\overline{\omega(z)}|}{1-|\omega(z)|^2}\right)^2\\
&\leq&(1+|z|)^2\frac{\gamma\cos\theta\sqrt{4+(\gamma^2+4\gamma)\cos^2\theta}}{2}+(1+|z|)|\omega(z)|\gamma \cos\theta\\[2mm]
&&+2(1+|z|)|\omega(z)|+\frac{3}{2}|\omega(z)|^2\\[2mm]
&\leq&(1+|z|)^2\frac{\gamma\cos\theta\sqrt{4+(\gamma^2+4\gamma)\cos^2\theta}}{2}+(1+|z|)|z|\gamma \cos\theta\\
&&+2(1+|z|)|z|+\frac{3}{2}|z|^2.\eeas
Therefore, the Schwarzian norm of $f\in\mathcal{HC}_\theta(\gamma)$ is 
\beas\Vert S_f\Vert=\sup_{z\in\mathbb{D}}\;(1-|z|^2)^2|S_f(z)|\leq 2\gamma\cos\theta\sqrt{4+(\gamma^2+4\gamma)\cos^2\theta}+2\gamma\cos\theta +11/2.\eeas
This completes the proof.
\end{proof}
\begin{rem}
The bound in Theorem 2.4 is not necessarily sharp. Finding the sharp bound for $\|S_f\|$ in $\mathcal{HC}_{\theta}(\gamma)$ remains an open problem, even in the special case $\theta = 0$.
\end{rem}
\begin{cor} For $|\theta|<\pi/2$ and $\gamma>0$, let $f=h+\ol{g}$ be in the class $\mathcal{HC}_\theta(\gamma)$ such that $g'(z)/h'(z)=z$. Then the Schwarzian norm satisfies the following inequality
\beas \Vert S_f\Vert \leq 2\gamma\cos\theta\sqrt{4+(\gamma^2+4\gamma)\cos^2\theta}+2\gamma\cos\theta +3/2.\eeas
\end{cor}
\begin{proof}Using argument as in the proof of \textrm{Theorem \ref{Th4}}, we obtain the Schwarzian derivative of $f=h+\ol{g}$ as follows:  
\beas&& S_f= S_h+\frac{\overline{z}}{1-|z|^2}\cdot\frac{h''(z)}{h'(z)}-\frac{3}{2}\left(\frac{\overline{z}}{1-|z|^2}\right)^2\\\text{\it i.e.,}
&&(1-|z|^2)^2 |S_f(z)|\leq  (1-|z|^2)^2|S_h|+|z|\cdot(1-|z|^2)\left|\frac{h''(z)}{h'(z)}\right|+\frac{3}{2}|z|^2.\eeas
Therefore, the Schwarzian norm of $f=h+\ol{g}$ is 
\beas\Vert S_f\Vert=\sup_{z\in\mathbb{D}}\;(1-|z|^2)^2|S_f(z)|\leq 2\gamma\cos\theta\sqrt{4+(\gamma^2+4\gamma)\cos^2\theta}+2\gamma\cos\theta +3/2.\eeas
This completes the proof.
\end{proof}
\section*{Declarations}
\noindent{\bf Acknowledgment:} The work of the second author is supported by University Grants Commission (IN) fellowship (No. F. 44 - 1/2018 (SA - III)).\\[2mm]
{\bf Conflict of Interest:} The authors declare that there are no conflicts of interest regarding the publication of this paper.\\[2mm]
{\bf Availability of data and materials:} Not applicable.\\[2mm]
{\bf Authors' contributions:} All authors contributed equally to the investigation of the problem, and all authors have read and approved the final manuscript.


\begin{thebibliography}{33}
\bibitem{AO2014}
{\sc R. Aghalary} and {\sc Z. Orouji}, Norm Estimates of the Pre-Schwarzian Derivatives for $\alpha$-Spiral-Like Functions of Order $\rho$, \textit{Complex Anal. Oper. Theory} \textbf{8}(4) (2014), 791--801.
\bibitem{ABM2025}{\sc V. Allu, R. Biswas} and {\sc R. Mandal}, An estimation of the pre-Schwarzian norm for certain classes of analytic functions, https://doi.org/10.48550/arXiv.2504.19132.
\bibitem{AP2023}{\sc M. F. Ali} and {\sc S. Pal}, Pre-Schwarzian norm estimates for the class of Janowski starlike functions, \textit{Monatsh. Math.} \textbf{201}(2) (2023), 311--327.
\bibitem{1AP2023} {\sc M. F. Ali} and {\sc S. Pal}, Schwarzian norm estimates for some classes of analytic functions, {\it Mediterr. J. Math.} {\bf20} (2023), 294.
\bibitem{2AP2023}{\sc M. F. Ali} and {\sc S. Pal}, Schwarzian norm estimate for functions in Robertson class, {\it Bull. Sci. Math.}, {\bf188} (2023), 103335.
\bibitem{AP2024}{\sc M. F. Ali} and {\sc S. Pal}, The Schwarzian norm estimates for Janowski convex functions, \textit{Proc. Edinb. Math. Soc.} {\bf67}(2) (2024), 299--315.
\bibitem{B1972} {\sc J. Becker}, L\"{o}wnersche Differentialgleichung und quasikonform fortsetzbare schlichte Funktionen, {\it J. Reine Angew. Math.} {\bf255} (1972), 23--43.
\bibitem{BP1984} {\sc J. Becker} and {\sc C. Pommerenke}, Schlichtheitskriterien und Jordangebiete, {\it J. Reine Angew. Math.} {\bf 354} (1984), 74--94.
\bibitem{B2025}{\sc R. Biswas}, Pre-Schwarzian norm estimation for functions in the Ma-Minda-type starlike and convex classes, https://doi.org/10.48550/arXiv.2505.01910.
\bibitem{CKPS2005}{\sc J. H. Choi, Y. C. Kim, S. Ponnusamy} and {\sc T. Sugawa}, Norm estimates for the Alexander transforms of convex functions of order alpha. \textit{J. Math. Anal. Appl.} \textbf{303} (2005), 661--668.
\bibitem{DP2008}{\sc S. Y.  Dai} and {\sc Y. F. Pan}, Note on Schwarz-Pick estimates for bounded and positive real part analytic functions, {\it Proc. Am. Math. Soc.} {\bf136} (2008), 635--640.
\bibitem{D1931}{\sc  J. A. Dieudonn\'e}, Recherches sur quelques probl\`emes relatifs aux polyn\^omes et aux fonctions born\'ees d'une variable complexe, {\it Ann. Sci. \'Ec. Norm. Sup\'er.} {\bf48} (1931), 247--358.
\bibitem{D1983} {\sc P. L. Duren},  Univalent functions, {\it Grundlehren der mathematischen Wissenschaften}, vol. {\bf259}, Springer-Verlag, New York, Berlin, Heidelberg, Tokyo, 1983.
\bibitem{D2004}{\sc P. Duren}, Harmonic Mapping in the Plane, {\it Cambridge University Press}, New York, 2004.
\bibitem{G2016}{\sc S. Y. Graf}, On the Schwarzian norm of harmonic mappings, {\it Probl. Anal. Issues Anal.} {\bf5}(23) (2016), 20--32.
 \bibitem{HM2015}{\sc R. Hern{\'a}ndez} and {\sc M. J. Mart\'in},  Pre-Schwarzian and Schwarzian derivatives of harmonic mappings, {\it J. Geomet. Anal.}, {\bf 25}(1)  (2015), 64--91.
 \bibitem{1HM2015}{\sc R. Hern{\'a}ndez} and {\sc M. J. Mart\'in}, Criteria for univalence and quasiconformal extension of harmonic mappings in terms of the Schwarzian derivative, {\it Arch. Math.} {\bf104} (2015), 53--59.
 \bibitem{H1949}{\sc E. Hille}, Remarks on a paper by Zeev Nehari, {\it Bull. Am. Math. Soc.} {\bf55} (1949), 552--553.
 \bibitem{JO1995}{\sc I. Jovanovi\'c} and {\sc M. Obradovi\'c}, A note on certain classes of univalent functions, {\it Filomat} {\bf9}(1) (1995), 69--72.
\bibitem{J1973} {\sc W. Janowski}, Extremal problems for a family of functions with positive real part and for some related families, {\it Ann. Polon. Math.} \textbf{23} (1973), 159--177 .
\bibitem{KPS2004} {\sc Y. C. Kim, S. Ponnusamy} and {\sc T. Sugawa}, Mapping properties of nonlinear integral operators and pre-Schwarzian derivatives, \textit{J. Math. Anal. Appl.} \textbf{299} (2004), 433--447.

\bibitem{KS2002} {\sc Y. C. Kim} and {\sc T. Sugawa}, Growth and coefficient estimates for uniformly locally univalent functions on the unit disk, {\it Rocky Mountain J. Math.} {\bf 32}  (2002), 179--200.

\bibitem{KS2006}{\sc Y. C. Kim} and {\sc T. Sugawa}, Norm estimates of the pre-Schwarzian derivatives for certain classes of univalent functions, \textit{Proc. Edinb. Math. Soc.} \textbf{49}(1) (2006), 131--143.
\bibitem{K1932} {\sc W. Kraus}, Uber den Zusammenhang eigner Characterstiken eines einfach zusammenhangenden Bereiches mit der Kreisabbildung, {\it Mitt. Math. Sem. Giessen}, {\bf 21}  (1932),  1--28.
\bibitem{K2025}{\sc S. Kumar}, Mapping properties of certain nonlinear integral operators involving Hornich operations, {\it Bull. Malays. Math. Sci. Soc.} {\bf48} (2025), 43.
\bibitem{L1987} {\sc O. Lehto}, Univalent functions and Teichm\"{u}ller spaces, Graduate Texts in Mathematics, vol. 109, Springer-Verlag, New York, 1987.
\bibitem{L1936} {\sc H. Lewy}, On the non-vanishing of the Jacobian in certain one-to-one mappings, {\it Bull. Am. Math. Soc.}, 42 (1936), 689--692.
\bibitem{LP2018}{\sc G. Liu} and {\sc S. Ponnusamy}, Uniformly locally univalent harmonic mappings associated with the pre-Schwarzian norm, {\it  Indag. Math.} {\bf29}(2) (2018), 752--778.
\bibitem{MM1992} {\sc W. Ma} and {\sc D. A. Minda}, Unified treatment of some special classes of univalent functions, In: Proceedings of the Conference on Complex Analysis, Tianjin. Conf. Proc. Lecture Notes Anal., I, Int. Press, Cambridge, MA, pp. 157--169 (1992).
\bibitem{N1949}{\sc Z. Nehari}, The Schwarzian derivative and schlicht functions, {\it Bull. Amer. Math. Soc.}, {\bf 55}(6) (1949), 545--551.
\bibitem{O2000}{\sc Y. Okuyama}, The norm estimates of pre-Schwarzian derivatives of spiral-like functions,  \textit{Complex Var. Theory Appl.} \textbf{42} (2000), 225--239.
\bibitem{O1998}{\sc  B. Osgood}, Old and New on the Schwarzian Derivative, In: P. Duren, J. Heinonen, B. Osgood, B. Palka (eds) Quasiconformal Mappings and Analysis, {\it Springer}, New York, 1998.
\bibitem{O1941}{\sc S. Ozaki}, On the theory of multivalent functions II, {\it Sci. Rep. Tokyo Bunrika Daigaku. Sect. A.} {\bf4} (1941), 45--87.
\bibitem{PPS2008}{\sc  R. Parvatham, S. Ponnusamy} and {\sc S. K. Sahoo}, Norm estimates for the Bernardi integral transforms of functions defined by subordination. \textit{Hiroshima Math. J.} \textbf{38} (2008), 19--29.
\bibitem{PR1995}{\sc S. Ponnusamy} and {\sc S. Rajasekaran}, New sufficient conditions for starlike and univalent functions, {\it Soochow J. Math.} {\bf21}(2) (1995), 193--201.
\bibitem{PS2008}{\sc S. Ponnusamy} and {\sc S. K. Sahoo}, Norm estimates for convolution transforms of certain classes of analytic functions, \textit{J. Math. Anal. Appl.} \textbf{342} (2008), 171--180.

\bibitem{PS2010}{\sc S. Ponnusamy} and {\sc S. K. Sahoo}, Pre-Schwarzian norm estimates of functions for a subclass of strongly starlike functions, \textit{Mathematica (Cluj)} \textbf{52}(75) (2010), 47--53.
\bibitem{S1973}{\sc G. M. Shah}, On holomorphic functions convex in one direction, {\it J. Indian Math. Soc.} {\bf37} (1973), 257--276.

\bibitem{S1966} {\sc J. Stankiewicz}, Quelques probl\`{e}mes extr\'{e}maux dans les classes des fonctions $\alpha$-angulairement \'{e}toil\'{e}es, \textit{Ann. Univ. Mariae Curie-Sk\l odowska Sect. A.} \textbf{20} (1966), 59--75.

\bibitem{S1998} {\sc T. Sugawa}, On the norm of pre-Schwarzian derivatives of strongly starlike functions, \textit{Ann. Univ. Mariae Curie-Sk\l odowska Sect. A} \textbf{52}(2) (1988), 149--157.
\bibitem{U1952}{\sc  T. Umezawa}, Analytic functions convex in one direction, {\it J. Math. Soc. Japan} {\bf4} (1952), 194--202.
\bibitem{UGS1994}{\sc B. A. Uralegaddi, M. D. Ganigi} and {\sc S. M. Sarangi}, Univalent functions with positive coefficients, {\it Tamkang J. Math.} {\bf25}(3) (1994), 225--230.

\bibitem{Y1976} {\sc S. Yamashita}, Almost locally univalent functions, \textit{Monatsh. Math.} \textbf{81} (1976), 235--240.


\bibitem{Y1999}{\sc S. Yamashita}, Norm estimates for function starlike or convex of order alpha, {\it Hokkaido Math. J.} {\bf 28}(1) (1999), 217--230.
\end{thebibliography}
\end{document}